\documentclass{article}
\usepackage{amsmath}
\usepackage{amssymb}
\usepackage{amsfonts}
\usepackage{amsthm}

\def\supp{\mathop{\mathrm{supp}}\nolimits}

\newtheorem{theorem}{Theorem}[section]
\newtheorem{remark}{Remark}[section]
\newtheorem{lemma}{Lemma}[section]

\newtheorem{corollary}{Corollary}[section]

\numberwithin{equation}{section}
\def\RR{\mathbb{R}}
\def\ZZ{\mathbb{Z}}

\def\NN{\mathbb{N}}

\newcommand{\vp}{v.p.}

\def\D{\,\mathrm d}
\def\e{\mathrm e}
\def\i{\mathrm i}

\def\rank{\mathop{\rm rank}\nolimits}

\title{ Construction of frames for shift-invariant spaces}

\author{Stevan Pilipovi\' c$^{1}$, Suzana Simi\' c$^{2}$}

\begin{document}
\date{}
\maketitle

\medskip

\medskip

\noindent
{\bf 2000 Mathematics Subject Classification}: 42C15, 42C40,
42C99, 46B15, 46B35, 46B20

\medskip

\noindent
{\bf Key Words and Phrases}:{$p$-frame; Banach frame; weighted shift-invariant
space.}

\begin{abstract}
We construct a sequence $\{\phi_i(\cdot-j)\mid j\in{\ZZ},\,i=1,\ldots,r \}$ which constitutes a $p$-frame for
  the weighted shift-invariant space
\[V^p_{\mu}(\Phi)=\Big\{\sum\limits_{i=1}^r\sum\limits_{j\in{\mathbb{Z}}}c_i(j)\phi_i(\cdot-j)
\;\Big|\;
\{c_i(j)\}_{j\in{\mathbb{Z}}}\in\ell^p_{\mu},\;i=1,\ldots,r\Big\},\;
p\in[1,\infty],\] and generates a closed shift-invariant subspace of $L^p_\mu(\mathbb{R})$.  The first construction  is obtained by choosing functions $\phi_i$, $i=1,\ldots,r$, with compactly supported Fourier transforms $\widehat{\phi}_i$, $i=1,\ldots,r$. The second construction,
with compactly supported $\phi_i,i=1,...,r,$ gives the Riesz basis.

\end{abstract}

\maketitle

\section{Introduction and preliminaries}\label{sec:1}
\indent

The shift-invariant spaces $V^p_\mu(\Phi)$, $p\in[1,\infty]$, quoted in the abstract, are used in the wavelet analysis, approximation theory, sampling theory, etc. They have been extensively studied in recent years by many authors \cite{ast}--\cite{ji3}).
The aim of this paper is to construct
$V^p_{\mu}(\Phi )$, $p\in[1,\infty]$, spaces with specially
chosen functions $\phi_i$, $i=1,\ldots,r$, which generate its
$p$-frame. These results expand and correct the construction obtained in \cite{pilsim}.
For the first construction, we take functions $\phi_i$, $i=1,\ldots,r$, so that the Fourier transforms are compactly supported smooth functions. Also, we derive the conditions for the collection
 $\{\phi_i(\cdot-k)\mid k\in{\ZZ},\, i=1,\ldots,r\}$ to
 form  a Riesz basis for $V^p_{\mu}(\Phi)$. We note that the properties of the
constructed
 frame guarantee the feasibility of a stable and continuous
 reconstruction algorithm in  $V^p_{\mu}(\Phi)$ \cite{xs}.
  We generalize these results for a shift-invariant subspace of $L^p_\mu(\RR^d)$.  The second construction is obtained by choosing compactly supported functions $\phi_i$, $i=1,\ldots,r$. In this way, we obtain the Riesz basis.

This paper is organized as follows. In Section~\ref{sec:2} we quote some
basic properties of certain subspaces of the  weighted $L^p$ and $\ell^p$
spaces. In Section~\ref{sec:3}
  we derive the conditions for the functions of the form $\widehat{\phi}_i(\xi)=\theta(\xi+k_i\pi)$, $k_i\in\mathbb{Z}$, $i=1,2,...,r$, $r\in\NN$,
to form a Riesz basis for $V^p_{\mu}(\Phi)$. We also show that using functions of the form $\widehat\phi_{i}(\cdot)=\theta(\cdot+i\pi)$, $i=1,\ldots,r$, where $\theta$ is compactly supported smooth function who's length of support is less than or equal to $2\pi$ we can not construct a $p$-frame for the shift-invariant space $V^p_{\mu}(\Phi)$.  In Section ~\ref{sec:4} we construct a sequence $\{\phi_i(\cdot-j)\mid j\in{\ZZ}^d,\, i=0,\ldots,r\}$, where $r\in2\NN$ or $r\in 3\NN$,  which constitutes a $p$-frame for
  the weighted shift-invariant space
$V^p_{\mu}(\Phi)$. Our construction shows that the sampling and
reconstruction problem in the shift-invariant spaces is robust in the sense of \cite{ast}. In Section~\ref{sec:5} we construct $p$-Riesz basis by using compactly supported functions $\phi_{i}$, $i=1,\ldots,r$.

\section{Basic spaces}\label{sec:2}

\indent

Let a function $\omega$ be nonnegative, continuous,
symmetric,
submultiplicative, i.e., $\omega(x+y)\leq\omega(x)\omega(y)$,
$x$, $y\in{\RR}^{d}$, and let a function $\mu $ be
$\omega$-moderate, i.e., $\mu  (x+y)\leq C\omega(x)\mu (y)$, $x$,
$y\in{\RR}^{d}$. Functions $\mu$ and $\omega$ are called weights.
We  consider the weighted function spaces $L^p_\mu$ and the weighted sequence spaces
$\ell^p_\mu({\ZZ}^d)$ with $\omega$-moderate weights $\mu$ (see\cite{pilsim}).
Let
$p\in[1,\infty)$. Then (with obvious modification for $p=\infty$)
$$\mathcal{L}^p_\omega= \Big\{f\;\big|\;
\|f\|_{\mathcal{L}^p_\omega}=\Big(\int\limits_{[0,1]^d}\Big(\sum\limits_{j\in{\ZZ}^d}|f(x+j)|\omega(x+j)\Big)
^pdx\Big)^{1/p}<+\infty\Big\},$$

$${W}^p_\omega:=\Big\{f\;\big|\;\|f\|_{{W}^p_\omega}=\Big(\sum\limits_{j\in{\ZZ}^d}\sup
\limits_{x\in[0,1]^d}|f(x+j)|^p\omega(j)^p\Big)^{1/p}<+\infty\Big\}.$$

In what follows, we use the notation
$\Phi=(\phi_1,\ldots,\phi_r)^T$. Define
$\|\Phi\|_{\mathcal{H}}=\sum\limits_{i=1}^r\|\phi_i\|_{\mathcal{H}}$,
where $\mathcal{H}=L^p_\omega$, $\mathcal{L}^p_\omega$ or
$W^p_\omega$,\; $p\in[1,\infty]$. With $\mathcal{F}\phi=\widehat{\phi}$ we denote the Fourier transform of the function $\phi$, i.e.\ $\widehat{\phi}(\xi)=\int_{\mathbb{R}^d}\phi(x)e^{-\i \pi x\cdot\xi} \D x$, $\xi\in\mathbb{R}^d$.

The concept of a $p$-frame is introduced in \cite{ast}:

It is said that a collection $\{\phi_i(\cdot-j)\mid j\in
{\ZZ}^d, i=1,\ldots,r\}$ is a $p$-frame for
$V^p_\mu(\Phi)$ if there exists a positive constant $C$ (dependant upon
$\Phi$, $p$ and $\omega$) such that
\begin{equation}\label{pframe}
C^{-1}\|f\|_{L^p_\mu}\leq\sum\limits_{i=1}^r\Big\|\Big\{\int\limits_{{\RR}^d}f(x)
\overline{\phi_i(x-j)}\D x\Big\}_{j\in{\ZZ}^d}\Big\|_{\ell^p_\mu}\leq
C\|f\|_{L^p_\mu},\quad f\in V^p_\mu(\Phi).
\end{equation}

Recall \cite{ac} that the shift-invariant spaces are defined by
$$V^p_\mu(\Phi):=\Big\{f\in L^p_\mu\mid
f(\cdot)=\sum\limits_{i=1}^r\sum\limits_{j\in{\ZZ}^d}{c^i_j}\,\phi_i(\cdot-j),
\;\;\{c^i_j\}_{j\in{\ZZ}^d}\in\ell^p_\mu,i=1,\ldots,r\Big\}.$$

\begin{remark}\cite{suza}
Let $\Phi\in
 W^1_\omega$ and let $\mu$ be $\omega$-moderate. Then $V^p_\mu(\Phi)$ is a subspace (not necessarily closed) of
$L^p_\mu$ and $W^p_\mu$ for any $p\in[1,\infty]$.
Clearly $(\ref{pframe})$ implies that $\ell^p_\mu$ and $V^p_\mu(\Phi)$
are isomorphic Banach spaces.
\end{remark}
Let $\Phi=(\phi_1,\ldots,\phi_r)^T$. Let $$[\widehat{\Phi},\widehat{\Phi}](\xi)=
\Big[\sum\limits_{k\in{\ZZ}^d}\widehat{\phi}_i(\xi+2k\pi)\overline{\widehat{\phi}_j(\xi+2k\pi)}\,\Big]_{1\leq
i\leq r,\;1\leq j\leq r},$$ where we assume that
$\widehat{\phi_i}(\xi)\overline{\widehat{\phi}_j(\xi)}$ is
integrable for any $1\leq i,j\leq r.$ Let
$A=[a(j)]_{j\in{\ZZ}^d}$ be an $r\times \infty$ matrix and
$A\overline{A^T}=\Big[\sum\limits_{j\in{\ZZ}^d}a_i(j)\overline{a_{i'}(j)}\Big]_{1\leq
i,i'\leq r}$. Then $ \rank A=\rank A\overline{A^T} $.

We will recall some results from \cite{ast} and \cite{pilsim} which are needed in the sequel.

\begin{lemma}[\cite{ast}]\label{lema1ald}
The following statements are equivalent.
\begin{itemize}
\item[$1)$]
$\rank\big[\widehat{\Phi}(\xi+2j\pi)\big]_{j\in{\ZZ}^d}$
 is a constant function on ${\RR}^d$.
\item[$2)$] $\rank[\widehat{\Phi},\widehat{\Phi}](\xi)$ is a
constant function on ${\RR}^d$. \item[$3)$] There exists a
positive constant $C$ independent of $\xi$ such that
$$C^{-1}[\widehat{\Phi},\widehat{\Phi}](\xi)\leq [\widehat{\Phi},\widehat{\Phi}]
(\xi)\,\overline{[\widehat{\Phi},\widehat{\Phi}](\xi)^T}\leq C\,[\widehat{\Phi},
\widehat{\Phi}](\xi), \quad \xi\in[-\pi,\pi]^d.$$
\end{itemize}
\end{lemma}

The next theorem (\cite{pilsim}) derives necessary and sufficient conditions for an indexed family $\{\phi_i(\cdot-j)\mid j\in{\ZZ}^d, i=1,\ldots,r\}$ to constitute a $p$-frame for $V^p_\mu(\Phi)$, which is equivalent with the closedness of this space in $L^p_\mu$. Thus, it is shown that under appropriate conditions on the frame vectors, there is an equivalence between the concept of $p$-frames, Banach frames  and the closedness of the space they generate.

\begin{theorem}[\cite{pilsim}]\label{main}
Let $\Phi=(\phi_1,\ldots,\phi_r)^T\in (W^1_\omega)^r$,
$p_0\in[1,\infty]$, and let $\mu$ be $\omega$-moderate. The
following statements are equivalent.
\begin{itemize}
\item[$i)$] \ $V^{p_0}_\mu(\Phi)$ is closed in $L^{p_0}_\mu$.
\item[$ii)$]\label{dva} \ $\{\phi_i(\cdot-j)\mid j\in{\ZZ}^d,i=1,\ldots,r\}$ is a $p_0$-frame for $V^{p_0}_\mu(\Phi)$.
\item[$iii)$] \ There exists a positive constant $C$ such that
$$C^{-1}[\widehat{\Phi},\widehat{\Phi}](\xi)\leq[\widehat{\Phi},
\widehat{\Phi}](\xi)\overline{[\widehat{\Phi},\widehat{\Phi}](\xi)^T}\leq
C[\widehat{\Phi},\widehat{\Phi}](\xi),\quad \xi\in[-\pi,\pi]^d.$$
\item[$iv)$] There exist positive constants $C_1$ and $C_2$ (depend on $\Phi$ and $\omega$) such that
\begin{equation}\label{cetiri}
C_1\|f\|_{L^{p_0}_\mu}\leq\inf\limits_{
f =\sum\limits_{i=1}^r\phi_i*'c^i}\sum\limits_{i=1}^r\|\{c^i_j\}_{j\in{\ZZ}^d}\|_{\ell^{p_0}_\mu}\leq
C_2\|f\|_{L^{p_0}_\mu},\;\;f\in V^{p_0}_\mu(\Phi).
\end{equation}
\item[$v)$] \ There exists
$\Psi=(\psi_1,\ldots,\psi_r)^T\in (W^1_\omega)^r$, such that
$$f=\sum\limits_{i=1}^r\sum\limits_{j\in{\ZZ}^d}\langle
f,\psi_i(\cdot-j)\rangle\phi_i(\cdot-j)=\sum\limits_{i=1}^r\sum\limits_{j\in{\ZZ}^d}
\langle f,\phi_i(\cdot-j)\rangle\psi_i(\cdot-j),\;\;f\in V^{p_0}_\mu(\Phi).$$
\end{itemize}
\end{theorem}

\begin{corollary}[\cite{pilsim}]\label{cor1}
Let $\Phi=(\phi_1,\ldots,\phi_r)^T\in (W^1_\omega)^r$,
$p_0\in[1,\infty]$, and let $\mu$ be $\omega$-moderate.
\begin{itemize}
\item[$i)$] \ If $\{\phi_i(\cdot-j)\mid j\in{\ZZ}^d,i=1,\ldots,r\}$ is a $p_0$-frame for $V^{p_0}_\mu(\Phi)$, then the collection
$\{\phi_i(\cdot-j)\mid j\in{\ZZ}^d,i=1,\ldots,r\}$ is a
$p$-frame for $V^p_\mu(\Phi)$ for any $p\in[1,\infty]$.
\item[$ii)$] \ If $V^{p_0}_\mu(\Phi)$ is closed in $L^{p_0}_\mu$
and $W^{p_0}_\mu$, then $V^p_\mu(\Phi)$ is closed in $L^p_\mu$ and
$W^p_\mu$ for any $p\in[1,\infty]$. \item[$iii)$] \ If
$(\ref{cetiri})$ holds for $p_0$, then it holds for any
$p\in[1,\infty]$.
\end{itemize}
\end{corollary}

\bigskip\bigskip

\section{Construction of frames using a compactly supported smooth function}\label{sec:3}

\indent

Considering the length of the support of a function $\theta$ and defining a function $\Phi$ in an appropriate way using $\theta$, we have different cases for the rank of the matrix $[\widehat{\Phi}(\xi+2j\pi)]_{j\in{\ZZ}}$ described in Theorem~\ref{te:nova2}.

First, we consider the next claim:\\
{\it Let  $\theta\in C^{\infty}_0(\RR)$ be a positive function such that
${\theta}(x)>0$, $x\in A$, $A\subset[-\pi,\pi]$, and $\supp\theta\subseteq[-\pi,\pi]$.
Moreover, let
$$\widehat{\phi}_k(\xi)={\theta}(\xi+k\pi),\quad k\in{\mathbb{Z}},
$$ and $\Phi=(\phi_i,\phi_{i+1},\ldots,\phi_{i+r})^T$, $i\in{\mathbb{Z}}$, $r\in{\mathbb{N}}$.

Then the rank of the matrix $[\widehat{\Phi}
(\xi+2j\pi)]_{j\in{\mathbb{Z}}}$ is not a constant function on ${\mathbb{R}}$ and it depends on $\xi\in{\mathbb{R}}$.}

As a matter of fact, by the Paley-Wiener theorem,
 $\phi_i\in\mathcal{S}(\mathbb{R})\subset W^1_\mu({\RR})$, $i\in\ZZ$.
For any $i\in\ZZ$, the matrix $[\widehat{\phi}_i,\widehat{\phi}_i](\xi)=\sum\limits_{j\in\ZZ}|\theta(\xi+i\pi+2j\pi)|^2$, $\xi\in \RR$, has the rank $0$ or $1$, depending on $\xi$. Moreover, we have $[\widehat{\phi}_{2i},\widehat{\phi}_{2i}](\pi)=0$ and $[\widehat{\phi}_{2i},\widehat{\phi}_{2i}](0)>0$.
Because of that, the rank of the matrix $[\widehat{\Phi}
(\xi+2j\pi)]_{j\in{\mathbb{Z}}}$ is not a constant function on ${\mathbb{R}}$ and it depends on $\xi\in{\mathbb{R}}$.

\begin{theorem}\label{te:nova2} Let  $\theta\in C^{\infty}_0(\RR)$ be a positive function such that
$\theta(x)>0$, $x\in(-\pi-\varepsilon,\pi+\varepsilon)$, and
 $\supp\theta=[-\pi-\varepsilon,\pi+\varepsilon],$ where $0<\varepsilon<1/4$.
Moreover, let
$$\widehat{\phi}_i(\xi)=\theta(\xi+k_i\pi),\quad k_i\in\mathbb{Z},\;i=1,2,...,r,\;r\in\NN,
$$ and $\Phi=(\phi_1,\phi_2,\ldots,\phi_r)^T$.

1) If  $|k_2-k_1|=2$ and $|k_i-k_j|\geq 2$ for different $i,j\leq r$, then the rank of the matrix $[\widehat{\Phi}
(\xi+2j\pi)]_{j\in{\mathbb{Z}}}$ is a constant function on ${\mathbb{R}}$ and equals $r$.

2) If $|k_2-k_1|=2$ and,  at least for  $k_{i_1}$ and $k_{i_2}$, it holds that $|k_{i_1}-k_{i_2}|=1$, where $1\leq i_1,i_2\leq r$, then the
rank of the matrix $[\widehat{\Phi}
(\xi+2j\pi)]_{j\in{\mathbb{Z}}}$ is a non-constant function on ${\mathbb{R}}$.
\end{theorem}

\begin{proof}
By the Paley-Wiener theorem,
 $\phi_i\in\mathcal{S}(\mathbb{R})\subset
W^1_\mu({\RR})$, $i=1,...,r$. All supporting cases are described in the following lemmas.

\begin{lemma}\label{le:pomocna1}
Let $\Phi=(\phi_{k_1},\phi_{k_2})^T$, $k_2-k_1=2$, $k_1,k_2\in{\ZZ}$. The rank of the matrix $[\widehat{\Phi}(\xi+2j\pi)]_{j\in{\ZZ}}$
is a constant function on ${\RR}$ and equals $2$.
\end{lemma}
\begin{proof}
We have the next two cases.

$1^\circ$ If $\xi\in(-\pi-\varepsilon-k_1\pi+2\ell\pi,-\pi+\varepsilon-k_1\pi+2\ell\pi)$, $\ell\in\ZZ$, for the matrix
$[\widehat{\Phi}
(\xi+2j\pi)]_{j\in{\mathbb{Z}}}$ we obtain $2\times\infty$ matrix
\[\left[\begin{array}{lllllll}
\cdots&     0&        0&      a^1&        b^1&      0&      \cdots\\
\cdots&     0&       a^2&     b^2&         0&      0&      \cdots
\end{array}\right],\] for some $0<a^i,b^i\leq1$, $i=1,2$. It is obvious that $\rank[\widehat{\Phi}
(\xi+2j\pi)]_{j\in{\mathbb{Z}}}=2$, $\xi\in(-\pi-\varepsilon-k_1\pi+2\ell\pi,-\pi+\varepsilon-k_1\pi+2\ell\pi)$, $\ell\in\ZZ$.

$2^\circ$ For $\xi\in[-\pi+\varepsilon-k_1\pi+2\ell\pi,\pi-\varepsilon-k_1\pi+2\ell\pi]$, $\ell\in\ZZ$, there are only two non-zero values $a^1$ and $a^2$ which are in different columns of the matrix $[\widehat{\Phi}
(\xi+2j\pi)]_{j\in{\mathbb{Z}}}$. Since
\[[\widehat{\Phi}
(\xi+2j\pi)]_{j\in{\mathbb{Z}}}=\left[\begin{array}{llllll}
\cdots&     0&         0&      a^1&       0&      \cdots\\
\cdots&     0&      a^2&      0&       0&      \cdots
\end{array}\right]_{2\times\infty},\] it
has the rank $2$ for  all $\xi\in[-\pi+\varepsilon-k_1\pi+2\ell\pi,\pi-\varepsilon-k_1\pi+2\ell\pi]$, $\ell\in\ZZ$.

We conclude that the rank of the matrix $[\widehat{\Phi}
(\xi+2j\pi)]_{j\in{\mathbb{Z}}}$, $\Phi=(\phi_{k_1},\phi_{k_2})^T$, $k_2-k_1=2$, $k_1,k_2\in{\ZZ}$, is a constant function on $\RR$ and equals $2$.

\end{proof}

\begin{lemma}\label{le:pomocna2}
The rank of the matrix $[\widehat{\Phi}(\xi+2j\pi)]_{j\in{\ZZ}}$ is not a constant function on ${\RR}$ if $\Phi=(\phi_{k_1},\phi_{k_2})^T$, $k_2-k_1=1$, $k_1,k_2\in{\ZZ}$.
\end{lemma}
\begin{proof}
In the same way, as in the proof of the Lemma~\ref{le:pomocna1}, we have four different cases for the matrix $[\widehat{\Phi}(\xi+2j\pi)]_{j\in{\ZZ}}$. Without losing generality, let we suppose that $k_1\in 2\ZZ$.

$1^\circ$ If $\xi\in(-\pi-\varepsilon+2\ell\pi,-\pi+\varepsilon+2\ell\pi)$, $\ell\in\ZZ$, then we have
\[[\widehat{\Phi}
(\xi+2j\pi)]_{j\in{\mathbb{Z}}}=
\left[\begin{array}{llllll}
\cdots&     0&        a^1&        b^1&      0&      \cdots\\
\cdots&     0&      a^2&     0&      0&      \cdots
\end{array}\right],\quad0<a^1,a^2,b^1\leq1,\]  and $\rank[\widehat{\Phi}
(\xi+2j\pi)]_{j\in{\mathbb{Z}}}=2$, for all $\xi\in(-\pi-\varepsilon+2\ell\pi,-\pi+\varepsilon+2\ell\pi)$, $\ell\in\ZZ$.

$2^\circ$ For $\xi\in[-\pi+\varepsilon+2\ell\pi,-\varepsilon+2\ell\pi]$, $\ell\in\ZZ$, non-zero values $a^1$ and $a^2$ are in the same column of the matrix  $[\widehat{\Phi}
(\xi+2j\pi)]_{j\in{\mathbb{Z}}}$. For any choice of a $2\times 2$
matrix, we get that the determinant equals $0$. So, we obtain
\[\rank\left[\begin{array}{lllll}
\cdots&         0&      a^1&      0&      \cdots\\
\cdots&         0&      a^2& 0& \cdots
\end{array}\right]=1,\] for all $\xi\in[-\pi+\varepsilon+2\ell\pi,-\varepsilon+2\ell\pi]$, $\ell\in\ZZ$.

$3^\circ$
If $\xi\in(-\varepsilon+2\ell\pi,\varepsilon+2\ell\pi)$, $\ell\in\ZZ$, then the matrix
\[[\widehat{\Phi}
(\xi+2j\pi)]_{j\in{\mathbb{Z}}}=\left[\begin{array}{llllllll}
\cdots& 0&      0&      a^1&      0&      \cdots\\
\cdots& 0&     b^2&      a^2& 0& \cdots
\end{array}\right],\] for some $0<a^1,a^2,b^2\leq1$, has the rank $2$, for all $\xi\in(-\varepsilon+2\ell\pi,\varepsilon+2\ell\pi)$, $\ell\in\ZZ$.

$4^\circ$
For $\xi\in[\varepsilon+2\ell\pi,\pi-\varepsilon+2\ell\pi]$, $\ell\in\ZZ$, there are two non-zero values $a^1$ and $b^2$  in different columns of the matrix $[\widehat{\Phi}
(\xi+2j\pi)]_{j\in{\mathbb{Z}}}$ and the block with these elements determines the rank $2$ for all $\xi\in[-\varepsilon+2\ell\pi,\pi-\varepsilon+2\ell\pi]$, $\ell\in\ZZ$.

Considering all cases, we conclude that the rank of the matrix $[\widehat{\Phi}
(\xi+2j\pi)]_{j\in{\mathbb{Z}}}$,  $\Phi=(\phi_{k_1},\phi_{k_2})^T$, $k_2-k_1=1$, $k_1,k_2\in{\ZZ}$, depends on $\xi\in\RR$ and equals $1$ or $2$. This rank is a non-constant function on $\RR$.
\end{proof}

{\it Proof of Theorem~\ref{te:nova2}}

$1)$ Using Lemma~\ref{le:pomocna1} and Lemma~\ref{le:pomocna2}, it is obvious that if $|k_2-k_1|=2$ and $|k_i-k_j|\geq 2$ for different $i,j\leq r$, then the position of the first non-zero element in each row of the matrix $[\widehat{\Phi}(\xi+2j\pi)]_{j\in{\ZZ}}$ is unique for each row. So, the  rank of the matrix $[\widehat{\Phi}
(\xi+2j\pi)]_{j\in{\ZZ}}$ is a constant function on $\RR$ and equals $r$  for all $\xi\in\RR$.

$2)$ If $|k_2-k_1|=2$ and, at least for  $k_{i_1}$ and $k_{i_2}$, it holds that $|k_{i_1}-k_{i_2}|=1$, $1\leq i_1,i_2\leq r$, then, in the row with the index $i_2$ (suppose, without losing generality, that $i_2\in2\ZZ+1$) we will have a new column with a non-zero element for $\xi\in(-\pi-\varepsilon+2\ell\pi,-\pi+\varepsilon+2\ell\pi)$, $\ell\in\ZZ$, but for $\xi\in[\varepsilon+2\ell\pi,\pi-\varepsilon+2\ell\pi]$, $\ell\in\ZZ$, the positions of all non-zero elements in that row will already appear in the previous columns. It is obvious that the rank of the matrix $[\widehat{\Phi}
(\xi+2j\pi)]_{j\in{\mathbb{Z}}}$  depends on $\xi\in\RR$ and is not the same for all $\xi\in\RR$.

\end{proof}

As a consequence of Theorem~\ref{main} and Theorem~\ref{te:nova2} 1), we have the following result.

\begin{theorem}
Let  the functions $\theta$ and $\Phi$ satisfy all the conditions of Theorem~\ref{te:nova2}~1).
Then the space $V^p_\mu(\Phi)$ is closed in $L^p_\mu$ for any
$p\in[1,\infty]$ and the family $\{\phi_{i}(\cdot -j)\mid
j\in{\ZZ},1\leq i\leq r\}$ is a $p$-Riesz basis for
$V^p_\mu(\Phi)$ for any $p\in[1,\infty]$.
\end{theorem}

The following theorem is a generalisation of Theorem~\ref{te:nova2} and can be proved in the same way, so we omit the proof.

\begin{theorem}
Let  $\theta\in C^{\infty}_0(\RR)$ be a positive function such that
$\theta(x)>0$, $x\in A$, $A\subset[a,b]$, $b>a$, and
supported by $[a,b]$ where $b-a>2\pi$.
Moreover, let
$$\widehat{\phi_i}(\xi)=\theta(\xi+k_i\pi),\quad k_i\in\mathbb{Z},\;i=1,2,...,r,\;r\in\NN,
$$ and $\Phi=(\phi_1,\phi_2,\ldots,\phi_r)^T$.

1) If  $|k_2-k_1|=2$ and $|k_i-k_j|\geq 2$ for different $i,j\leq r$, then the rank of the matrix $[\widehat{\Phi}
(\xi+2j\pi)]_{j\in{\mathbb{Z}}}$ is a constant function on ${\mathbb{R}}$ and equals $r$.

2) If $|k_2-k_1|=2$ and, at least for  $k_{i_1}$ and $k_{i_2}$, it holds that $|k_{i_1}-k_{i_2}|=1$, where $1\leq i_1,i_2\leq r$, then the
rank of the matrix $[\widehat{\Phi}
(\xi+2j\pi)]_{j\in{\mathbb{Z}}}$ is not a constant function on ${\mathbb{R}}$.
\end{theorem}

\bigskip\bigskip

\section{Construction of frames using several compactly supported smooth functions}\label{sec:4}

\indent

Firstly, we consider two smooth functions with proper compact supports.

\begin{lemma}\label{te:nova3} Let  $\theta\in C^{\infty}_0(\RR)$, $\psi\in C^{\infty}_0(\RR)$ be positive functions such that
\begin{align*}{\theta}(x)>0,\;& x\in(-\varepsilon,2\pi+\varepsilon),\qquad\;\; \supp{\theta}=[-\varepsilon,2\pi+\varepsilon],\\
{\psi}(x)>0,\;& x\in(\varepsilon,2\pi-\varepsilon),\qquad\;\; \;\;\supp{\psi}=[\varepsilon,2\pi-\varepsilon],\quad 0<\varepsilon<1/4.\end{align*}
Moreover, let $\widehat{\phi}_1(\xi)=\theta(\xi)$, $\widehat{\phi}_2(\xi)=\psi(\xi)$, $\xi\in\RR$, and
$\Phi=(\phi_1,\phi_2)^T$. Then the rank of the matrix $[\widehat{\Phi}
(\xi+2j\pi)]_{j\in{\mathbb{Z}}}$ is a constant function on ${\mathbb{R}}$ and equals $1$.
\end{lemma}

\begin{proof}
Note that
 $\phi_i\in\mathcal{S}(\mathbb{R})\subset
W^1_\mu({\RR})$, $i=1,2$.

We have the following two cases.

$1^\circ$ If $\xi\in(-\varepsilon+2\ell\pi,\varepsilon+2\ell\pi)$, $\ell\in\ZZ$, then the matrix
\[[\widehat{\Phi}
(\xi+2j\pi)]_{j\in{\mathbb{Z}}}=\left[\begin{array}{llllll}
\cdots&     0&      a&      b&      0&      \cdots\\
\cdots&     0&      0&      0&      0&      \cdots
\end{array}\right],\quad 0<a,b\leq1,\] has a constant rank and equals $1$.

$2^\circ$ For $\xi\in(\varepsilon+2\ell\pi,2\pi-\varepsilon+2\ell\pi)$, $\ell\in\ZZ$, the rank of the matrix
\[[\widehat{\Phi}
(\xi+2j\pi)]_{j\in{\mathbb{Z}}}=\left[\begin{array}{llllll}
\cdots&     0&      c&      0&      \cdots\\
\cdots&     0&      d&      0&      \cdots
\end{array}\right],\] where $c,d$ are non-zero values, equals $1$. The equivalent matrix is obtained for $\xi=\varepsilon+2\ell\pi$ and $\xi=-\varepsilon+2\ell\pi$, so we conclude that
 $\rank[\widehat{\Phi}
(\xi+2j\pi)]_{j\in{\mathbb{Z}}}=1$, for $\xi\in[\varepsilon+2\ell\pi,2\pi-\varepsilon+2\ell\pi]$, $\ell\in\ZZ$.

Considering these two cases, the rank of the matrix $[\widehat{\Phi}
(\xi+2j\pi)]_{j\in{\mathbb{Z}}}$ is a constant function on $\RR$ and $\rank[\widehat{\Phi}
(\xi+2j\pi)]_{j\in{\mathbb{Z}}}=1$, $\xi\in\RR$.

\end{proof}

Using functions $\theta$ and $\psi$ from Lemma~\ref{te:nova3}, in the next lemma we construct the $p$-frame with four appropriate functions.

\begin{lemma}\label{te:nova4} Let the functions $\theta$ and $\psi$ satisfy all the conditions of Lemma~\ref{te:nova3}.
Moreover, let
$$\widehat{\phi}_k(\xi)={\theta}(\xi+2k\pi), \quad \widehat{\phi}_{k+2}(\xi)={\psi}(\xi+2k\pi),\quad \;k=0,1,$$ and $\Phi=(\phi_0,\phi_1,\phi_2,\phi_3)^T$.

The rank of the matrix $[\widehat{\Phi}
(\xi+2j\pi)]_{j\in{\mathbb{Z}}}$ is a constant function on ${\mathbb{R}}$ and equals $2$.
\end{lemma}

\begin{proof}
The proof is similar to the proof of Lemma~\ref{te:nova3}.

$1^\circ$ If $\xi\in(-\varepsilon+2\ell\pi,\varepsilon+2\ell\pi)$, $\ell\in\ZZ$, then the matrix
\[[\widehat{\Phi}
(\xi+2j\pi)]_{j\in{\mathbb{Z}}}=\left[\begin{array}{ccccccc}
\cdots&     0&      0&      a^1&      b^1&      0&      \cdots\\
\cdots&     0&      a^2&      b^2&      0&      0&      \cdots\\
\cdots&     0&      0&      0&      0&      0&      \cdots\\
\cdots&     0&      0&      0&      0&      0&      \cdots
\end{array}\right],\] where $0<a^i,b^i\leq1$, $i=1,2$, has a constant rank and equals $1$.

$2^\circ$ For $\xi\in[\varepsilon+2\ell\pi,2\pi-\varepsilon+2\ell\pi]$, $\ell\in\ZZ$, we have
\[\rank[\widehat{\Phi}
(\xi+2j\pi)]_{j\in{\mathbb{Z}}}=\left[\begin{array}{ccccccc}
\cdots&     0&     0&      c^1&      0&      \cdots\\
\cdots&     0&     0&      d^1&      0&      \cdots\\
\cdots&     0&     c^2&      0&      0&      \cdots\\
\cdots&     0&     d^2&      0&      0&      \cdots
\end{array}\right]=2,\] where $0<c^i,d^i\leq1$, $i=1,2$.

We conclude that the rank of the matrix $[\widehat{\Phi}
(\xi+2j\pi)]_{j\in{\mathbb{Z}}}$ is a constant function on ${\mathbb{R}}$ and equals $2$.
\end{proof}

Lemma~\ref{te:nova3} can be easily generalised  for an even number of functions $\phi_i$, $i=0,\ldots,2r-1$, with compactly supported $ \widehat{\phi}_i$, $i=0,\ldots,2r-1$. The proof of the next theorem is similar to the previous proofs.

\begin{theorem}\label{te:nova5} Let the functions $\theta$ and $\psi$ satisfy all the conditions of Lemma~\ref{te:nova3}.
Moreover, let
$$\widehat{\phi}_k(\xi)=\theta(\xi+2k\pi), \quad \widehat{\phi}_{k+r}(\xi)={\psi}(\xi+2k\pi),\quad \;k=0,\ldots,r-1,\;r\in\NN,$$ and $\Phi=(\phi_0,\phi_1,\ldots,\phi_{2r-1})^T$.

The following statements hold.
\begin{itemize}
\item[$1^\circ$] \
$\rank[\widehat{\Phi}(\xi+2j\pi)]_{j\in{\mathbb{Z}}}=r$ for all
$\xi\in{\mathbb{R}}$.

\item[$2^\circ$] \ $V^p_\mu(\Phi)$ is closed in $L^p_\mu$ for any
$p\in[1,\infty]$.

\item[$3^\circ$] \ $\{\phi_{i}(\cdot -j)\mid
j\in{\mathbb{Z}},0\leq  i\leq  2r-1\}$ is a $p$-frame for
$V^p_\mu(\Phi)$ for any $p\in[1,\infty]$.
\end{itemize}
\end{theorem}

Now we consider three functions with compact supports.

\begin{lemma}\label{te:nova6} Let  the function $\theta$  satisfies all the conditions of Lemma~\ref{te:nova3}, and let $\tau\in C^{\infty}_0(\RR)$ and $\omega\in C^{\infty}_0(\RR)$ be positive functions such that
\begin{align*}{\tau}(x)>0,\;& x\in(\varepsilon,\pi-\varepsilon)\cup(\pi+\varepsilon,2\pi-\varepsilon),\quad \supp{\tau}=[\varepsilon,\pi-\varepsilon]\cup[\pi+\varepsilon,2\pi-\varepsilon],\\
{\omega}(x)>0,\;& x\in(-3\pi-\varepsilon,-\pi+\varepsilon),\quad\supp{\omega}=[-3\pi-\varepsilon,-\pi+\varepsilon],\quad 0<\varepsilon<1/4.\end{align*}
Moreover, let $\widehat{\phi}_1(\xi)=\theta(\xi)$, $\widehat{\phi}_2(\xi)=\tau(\xi)$, $\widehat{\phi}_3(\xi)=\omega(\xi)$, $\xi\in\RR$, and $\Phi=(\phi_1,\phi_2,\phi_3)^T$. Then
the rank of the matrix $[\widehat{\Phi}
(\xi+2j\pi)]_{j\in{\mathbb{Z}}}$ is a constant function on ${\mathbb{R}}$ and equals $2$.
\end{lemma}

\begin{proof}
We have four different forms for the matrix $[\widehat{\Phi}
(\xi+2j\pi)]_{j\in{\mathbb{Z}}}$  and in all cases the rank of the matrix is $2$.

Now we will show all possible cases. Denote with $a^i$, $i=1,2,3$ and $b^i$, $i=1,2$, some positive values.

$1^\circ$
\[[\widehat{\Phi}
(\xi+2j\pi)]_{j\in{\mathbb{Z}}}=\left[\begin{array}{cccccccc}
\cdots&     0&     a^1&      b^1&      0&       0&     \cdots\\
\cdots&     0&     0&      0&      a^2&         0&     \cdots\\
\cdots&     0&     0&      0&      0&       0&     \cdots
\end{array}\right],\quad \xi\in(-\varepsilon+2\ell\pi,\varepsilon+2\ell\pi).\]

$2^\circ$
\[[\widehat{\Phi}
(\xi+2j\pi)]_{j\in{\mathbb{Z}}}=\left[\begin{array}{cccccccc}
\cdots&     0&      b^1&      0&       0&     \cdots\\
\cdots&     0&      0&      a^2&         0&     \cdots\\
\cdots&     0&      a^3&      0&       0&     \cdots
\end{array}\right],\quad \xi\in[\varepsilon+2\ell\pi,\pi-\varepsilon+2\ell\pi].\]

$3^\circ$
\[[\widehat{\Phi}
(\xi+2j\pi)]_{j\in{\mathbb{Z}}}=\left[\begin{array}{ccccccccc}
\cdots&     0&      b^1&      0&       0&0&      \cdots\\
\cdots&     0&      0&      a^2&         b^2& 0&     \cdots\\
\cdots&     0&      0&      0&       0&     0& \cdots
\end{array}\right],\quad \xi\in(\pi-\varepsilon+2\ell\pi,\pi+\varepsilon+2\ell\pi).\]

$4^\circ$
\[[\widehat{\Phi}
(\xi+2j\pi)]_{j\in{\mathbb{Z}}}=\left[\begin{array}{ccccccccc}
\cdots&     0&      b^1&      0&       0&0&      \cdots\\
\cdots&     0&      0&      0&         b^2& 0&     \cdots\\
\cdots&     0&      a^3&      0&       0&     0& \cdots
\end{array}\right],\quad \xi\in[\pi+\varepsilon+2\ell\pi,2\pi-\varepsilon+2\ell\pi].\]

\end{proof}

\begin{remark}
In Lemma~\ref{te:nova6} the support of the function $\omega$ must have an empty intersection with the supports of $\theta$ and $\tau$. In the opposite case, i.e. $\supp{\theta}\cap\supp{\tau}\cap\supp{\omega}\neq\emptyset$,  the rank of the matrix $[\widehat{\Phi}
(\xi+2j\pi)]_{j\in{\mathbb{Z}}}$ is a non-constant function on ${\mathbb{R}}$.
\end{remark}

Lemma~\ref{te:nova6} can be easily generalised  for functions $\phi_i$, $i=0,\ldots,3r-1$, with compactly supported $\widehat{\phi}_i$, $i=0,\ldots,3r-1$. The proof of the next theorem is similar to the previous proofs.

\begin{theorem}\label{te:nova7}  Let the functions $\theta$, $\tau$ and $\omega$ satisfy all the conditions of Lemma~\ref{te:nova6}.
Moreover, let
$$\widehat{\phi}_k(\xi)=\theta(\xi+2k\pi), \; \widehat{\phi}_{k+r}(\xi)=\tau(\xi+2k\pi),\; \widehat{\phi}_{k+2r}(\xi)=\omega(\xi+2k\pi),$$ $k=0,\ldots,r-1$,  $r\in\NN$, and $\Phi=(\phi_0,\phi_1,\ldots,\phi_{3r-1})^T$.

The following statements hold.
\begin{itemize}
\item[$1^\circ$] \
$\rank[\widehat{\Phi}(\xi+2j\pi)]_{j\in{\mathbb{Z}}}=2r$ for all
$\xi\in{\mathbb{R}}$.

\item[$2^\circ$] \ $V^p_\mu(\Phi)$ is closed in $L^p_\mu$ for any
$p\in[1,\infty]$.

\item[$3^\circ$] \ $\{\phi_{i}(\cdot -j)\mid
j\in{\mathbb{Z}},0\leq  i\leq  3r-1\}$ is a $p$-frame for
$V^p_\mu(\Phi)$ for any $p\in[1,\infty]$.
\end{itemize}
\end{theorem}
\section{Construction of frames of functions with finite regularities and compact supports; one-dimensional case}\label{sec:5}

\bigskip
\indent

Let $H(x)$, $x\in\mathbb{R}$, be the characteristic function of the semiaxis $x\geq0$, i.e.\ $H(x)=0$ if $x<0$ and $H(x)=1$ if $x\geq 0$ (Heaviside's function).
We construct a sequence $\{\phi_n\}_{n\in\mathbb{N}}$ in the following way. Let $\phi_1(x):=(H(x)-H(x-a))/a$, $a>0$, $\phi_2:=\phi_1*\phi_1$, $\phi_3:=\phi_1*\phi_1*\phi_1$, $\ldots$, i.e., \[ {\phi_n:=\underbrace{\phi_1*\phi_1*\cdots*\phi_1}\limits_{n-1\;\mbox{times}}}, \quad n\in\mathbb{N}, \] where $*$ denotes the convolution of the functions.

 We obtain
 \begin{eqnarray*}\phi_2(x)&=&\frac{1}{a^2}\Bigl(x H(x)-2(x-a)H(x-a)+(x-2a)H(x-2a)\Bigr),\\
  \phi_3(x)&=&\frac{1}{2!a^3}\Bigl(x^2 H(x)-3(x-a)^2H(x-a)\\
  &&\quad\quad+3(x-2a)^2H(x-2a)-(x-3a)^2H(x-3a)\Bigr),\\
    \phi_4(x)&=&\frac{1}{3!a^4}\Bigl(x^3 H(x)-4(x-a)^3H(x-a)+6(x-2a)^3H(x-2a)\\
  &&\quad\quad-4(x-3a)^3H(x-3a)+(x-4a)^3H(x-4a)\Bigr).
    \end{eqnarray*} Continuing in this manner, for all $n\in\mathbb{N}$, we have \begin{eqnarray*}\phi_n(x)&=&\frac{1}{a^n(n-1)!}\biggl({n\choose0}x^{n-1} H(x)-{n\choose1}(x-a)^{n-1}H(x-a)\\
  &&\quad\quad+{n\choose2}(x-2a)^{n-1}H(x-2a)-{n\choose3}(x-3a)^{n-1}H(x-3a)\\
  &&\quad\quad+\cdots+(-1)^{n-1}{n\choose{n-1}}
    (x-(n-1)a)^{n-1}H(x-(n-1)a)\\
  &&\quad\quad+(-1)^n{n\choose n}(x-na)^{n-1}H(x-na)\biggr).\end{eqnarray*}
Calculating the Fourier transform of functions $\phi_n$, $n\in\mathbb{N}$, we get
\begin{eqnarray*}\widehat{\phi}_1(\xi)&=&\frac{-\i}{a}\;\vp\Bigl(\frac{1}{\xi}\Bigr)(\e^{\i a\xi}-1),\\
\widehat{\phi}_2(\xi)&=&\frac{(-\i)^2}{a^2}\;\vp\Bigl(\frac{1}{\xi^2}\Bigr)(\e^{\i a\xi}-1)^2,\\
\widehat{\phi}_3(\xi)&=&\frac{(-\i)^3}{a^3}\;\vp\Bigl(\frac{1}{\xi^3}\Bigr)(\e^{\i a\xi}-1)^3.\end{eqnarray*} Continuing in this manner, we obtain
$\widehat{\phi}_n(\xi)=\frac{(-\i)^n}{a^n}\;\vp\Bigl(\frac{1}{\xi^n}\Bigr)(\e^{\i a\xi}-1)^n$, $n\in\mathbb{N}$,
where $\vp$ denotes the principal value.

Let $\Phi=(\phi_1,\phi_2,\ldots,\phi_r)^T$, $r\in\mathbb{N}$. The matrix $[\widehat{\Phi}
(\xi+2j\pi)]_{j\in{\ZZ}}$ has for all $\xi\in\mathbb{R}$ the same rank as the matrix

$$R(\xi)=\left[\begin{array}{lllllllllllll}
&\cdots &\alpha_{-4\pi}\beta_{-4\pi}&\alpha_{-2\pi}\beta_{-2\pi}&\alpha_{0}\beta_{0} &\alpha_{2\pi}\beta_{2\pi} &\alpha_{4\pi}\beta_{4\pi}&\cdots \\

&\cdots &\alpha^2_{-4\pi}\beta^2_{-4\pi}&\alpha^2_{-2\pi}\beta^2_{-2\pi}&\alpha^2_{0}\beta^2_{0} &\alpha^2_{2\pi}\beta^2_{2\pi} &\alpha^2_{4\pi}\beta^2_{4\pi}&\cdots\\

&\cdots &\alpha^3_{-4\pi}\beta^3_{-4\pi}&\alpha^3_{-2\pi}\beta^3_{-2\pi}&\alpha^3_{0}\beta^3_{0} &\alpha^3_{2\pi}\beta^3_{2\pi} &\alpha^3_{4\pi}\beta^3_{4\pi}&\cdots\\

&\cdots &\alpha^4_{-4\pi}\beta^4_{-4\pi}&\alpha^4_{-2\pi}\beta^4_{-2\pi}&\alpha^4_{0}\beta^4_{0} &\alpha^4_{2\pi}\beta^4_{2\pi} &\alpha^4_{4\pi}\beta^4_{4\pi}&\cdots\\

&\; &\;\;\;\;\vdots&\;\;\;\;\vdots&\;\;\;\;\vdots&\;\;\;\;\vdots&\;\;\;\;\vdots&\; \\

&\cdots &\alpha^r_{-4\pi}\beta^r_{-4\pi}&\alpha^r_{-2\pi}\beta^r_{-2\pi}&\alpha^r_{0}\beta^r_{0} &\alpha^r_{2\pi}\beta^r_{2\pi} &\alpha^r_{4\pi}\beta^r_{4\pi}&\cdots
\end{array}\right],$$ where $\alpha^m_k=\vp\Bigl(\frac{1}{\xi-k}\Bigr)^m$ and $\beta^m_k=\bigl(\e^{\i a(\xi-k)}-1\bigr)^m$.
Since the rank of $R(\xi)$ is equal to $r$ for all $\xi\in\mathbb{R}$, we have the next result.
\begin{theorem}\label{tedrugakons}
Let $\Phi=(\phi_k,\phi_{k+1},\ldots,\phi_{k+(r-1)})^T$, for $k\in\ZZ$,
$r\in{\NN}$. Then $V^p_\mu(\Phi)$ is closed in $L^p_\mu$ for any
$p\in[1,\infty]$ and $\{\phi_{k+s}(\cdot -j)\mid
j\in{\ZZ},0\leq s\leq r-1\}$ is a $p$-Riesz basis for
$V^p_\mu(\Phi)$ for any $p\in[1,\infty]$.
\end{theorem}
\begin{remark} $(1)$ We refer to \cite{2ald}  and \cite{xs} for the
$\gamma$-dense set $X=\{x_j\mid j\in J\}$. Let
$\phi_k(x)=\mathcal{F}^{-1}(\theta(\cdot-k\pi))(x)$,
$x\in{\mathbb{R}}$. Following the notation of \cite{xs}, we put
$\psi_{x_j}=\phi_{x_j}$ where $\{x_j\mid j\in J\}$ is
$\gamma$-dense set determined by $f\in
V^2(\phi)=V^2(\mathcal{F}^{-1}(\theta))$.
Theorems 3.1, 3.2 and 4.1 in \cite{xs} give the
conditions and explicit form of $C_p>0$ and $c_p>0$ such that the inequality
$c_p\|f\|_{L^p_\mu}\le \Big(\sum\limits_{j\in J}|\langle
f,\psi_{x_j}\rangle\mu(x_j)|^p\Big)^{1/p}\le
C_p\|f\|_{L^p_\mu}$ holds. This inequality guarantee the
feasibility of a stable and continuous reconstruction algorithm in
the signal spaces $V^p_\mu(\Phi)$ (\cite{xs}).

$(2)$  Since the
 spectrum of the Gram matrix
 $[\widehat{\Phi},\widehat{\Phi}](\xi)$, where $\Phi$ is defined in Theorem \ref{tedrugakons},  is bounded and bounded
 away from zero  (see \cite{bor}), it follows that  the family
 $\{\Phi(\cdot-j)\mid j\in{\mathbb{Z}}\}$
 forms a $p$-Riesz basis for $V^p_{\mu}(\Phi)$.

$(3)$ Frames of the above sections may be useful in
 applications since they satisfy assumptions of Theorem $3.1$ and
 Theorem $3.2$ in \cite{akr}. They show  that error analysis for
 sampling and reconstruction can be tolerated, or that the
 sampling and reconstruction problem in shift-invariant space is
 robust with respect to appropriate set of functions
 $\phi_{k_1},\ldots,\phi_{k_r}$.

\end{remark}

\section*{Acknowledgment}
\indent
The authors were supported in part by the Serbian
Ministry of Science and Technological Developments (Project
$\#$174024).

\end{document}